\documentclass[11pt]{amsart}
\usepackage{graphicx}
\usepackage{latexsym}
\usepackage{amsfonts,amsmath,amssymb}

\usepackage{color}

\def\pr{\textup{ P\/}}
\def\ex{\textup{E\/}}

\def\eps{\varepsilon}
\def\la{\lambda}

\def\part{\partial}

\newcommand{\beq}{\begin{equation}}
\newcommand{\eeq}{\end{equation}}

\newtheorem{Theorem}{Theorem}[section]
\newtheorem{Lemma}[Theorem]{Lemma}

\newtheorem{Corollary}[Theorem]{Corollary}

\theoremstyle{remark}

\numberwithin{equation}{section}
\date{\today}
\begin{document}

\title[Local maxima supports]{On random quadratic forms: supports of potential local maxima}
\author{Boris Pittel}
\address{Department of Mathematics, The Ohio State University, Columbus, Ohio 43210, USA}
\email{bgp@math.ohio-state.edu}



\begin{abstract} The selection model in population genetics is a dynamic system on the set of
of the probability distributions $\bold p=(p_1,\dots,p_n)$ of the alleles $A_1,\dots, A_n$, with 
$p_i(t+1)$ proportional to $p_i(t)$ times $\sum_j f_{i,j}p_j(t)$, and $f_{i,j}=f_{j,i}$ interpreted as a fitness
of the gene pair $(A_i,A_j)$. It is known that $\hat {\bold p}$ is a locally stable equilibrium iff $\hat{\bold p}$
 is a strict local maximum of the quadratic form $\bold p^T\bold f\,\bold p$. Usually there are multiple
local maxima and $\lim\bold p(t)$ depends on $\bold p(0)$. To address the question of a
{\it typical\/} behavior of $\{\bold p(t)\}$, John Kingman considered the case when 
the $f_{i,j}$ are independent, $[0,1]$-uniform. He proved
 that with high probability (w.h.p.) no local maximum may have more than $2.49 n^{1/2}$ positive
 components, and reduced $2.49$ to $2.14$ for a non-biological case of 
exponentials on $[0,\infty)$. We show that the constant $2.14$ serves a broad class of the smooth densities 
on $[0,1]$ with the increasing hazard rate. As for a lower bound, we prove that w.h.p. for all $k\le 2n^{1/3}$ 
there are many $k$-element subsets of $[n]$ that pass a partial test to be a support of a local maximum. Still it may well be
that w.h.p. the actual supports are much smaller. In that direction we prove that  w.h.p. {\bf (i)\/} a support of a local maximum,
that does not contain a support of a local equilibrium, is very unlikely to have size exceeding $(2/3)\log_2 n$, and {\bf (ii)\/}
for the uniform fitnesses, there are super-polynomially many potential supports free of local equilibriums, of size close to $(1/2)\log_2 n$.
\end{abstract}

\maketitle

\section{Introduction and main results}  The classic selection model in population genetics is a dynamic system on the set 
of the probability distributions $\bold p=(p_1,\dots,p_n)\in \Delta_n:=\{\bold x\ge \bold 0,\,\sum_{i\in [n]}x_i=1\}$ of the alleles $A_1,\dots, A_n$ at the single locus:
\begin{equation}\label{basicrec}
p_i(t+1)=p_i(t)\cdot\frac{\sum_jf_{i,j}\,p_j(t)}{\sum_{r,s}f_{r,s}\,p_r(t) p_s(t)},\quad i\in [n].
\end{equation}
Here each $f_{r,s}=f_{s,r}\in [0,1]$ is interpreted as the fitness, i.e. the probability that the unordered gene pair $(A_r,A_s)$
survives to an adult age. While the dynamic behavior of $\bold p(t)$ in this model certainly depends on the fitness matrix $\bold f=\{f_{r,s}\}$, it has long been known that the {\it average fitness\/} $V(\bold p(t)):=\sum_{r,s}f_{r,s}\,p_r(t) p_s(t)$ strictly increases with $t$ unless $\bold p(t+1)=\bold p(t)$. Hofbauer and Sigmund \cite{HofSig} {\bf (i)\/} characterized this property as a consequence of Fisher's 
Fundamental Theorem of Natural Selection \cite{Fis}, {\bf (ii)\/}  provided a full proof following Kingman \cite{Kin}, and {\bf (iii)\/} sketched the different proofs given  
by Scheuer and Mandel \cite{SchMan}, and Baum and Eagon \cite{BauEag}.

Using the increase of the average fitness, it was later proven by various researchers that  {\bf (i)\/} $\bold {p}(\infty)=\lim_{t\to\infty}\bold p(t)$ exists for every initial
gene distribution $\bold p(0)$, and {\bf (ii)\/} $\bold p:=\bold p(\infty)$ is a fixed point of the mapping $\boldsymbol\Phi(\cdot):\,\Delta_n\to\Delta_n$ defined by \eqref{basicrec}, with a property: for a nonempty $I\subseteq [n]$, 
\begin{equation}\label{I,V}
p_i=0,\,\,(i\notin I),\quad \sum_{j\in I}f_{i,j}p_j\equiv V(\bold p),\,\,(i\in I).
\end{equation}
Remarkably, a fixed point $\bold p$ is a {\it locally\/} stable equilibrium iff $\bold p$ is a strict local
maximum of $V(\cdot)$. There is no reason to expect that a local maximum is unique; so typically the limit $\bold p(\infty)$ depends
on $\bold p(0)$.

For $\bold p\in \Delta_n$ to be a local maximum of $V(\cdot)$, three sets of conditions must be satisfied, Kingman \cite{Kin1}. If $I=I(\bold p):=\{i: p_i>0\}$, then
\begin{equation}\label{King1,2,3}
\begin{aligned}
\sum_{j\in I}f_{i,j}p_j&\equiv V(\bold p),\quad (i\in I),\\
\sum_{i,j\in I}f_{i,j} x_ix_j&\le 0,\quad\forall\, \{x_i\}_{i\in I}\,\text{  with } \sum_{i\in I}x_i=0,\\
\sum_{j\in I}f_{i,j}p_j&\le V(\bold p),\quad (i\notin I).
\end{aligned}
\end{equation}
The second necessary condition applied  to $\bold x$ such that $x_i=1$, $x_j=-1$, with the remaining 
$x_k=0$, easily yields
\begin{equation}\label{King2'}
f_{i,j}\ge \frac{f_{i,i}+f_{j,j}}{2},\quad i,j\in I,\,i\neq j.
\end{equation}
To quote from \cite{Kin1}: ``This condition uses internal stability alone, and takes no account of vulnerability
to mutation''. 

The inequality \eqref{King2'} was earlier obtained by Lewontin, Ginzburg and Tuljapurkar  as a corollary of a determinantal criterion applied to the system of $(|I|-1)$ linear equations for $p_i$, $i\in I\setminus\{i_0\}$, implicit in 
\[
\sum_{j\in I}f_{i,j}p_j\equiv V(\bold P),\,\,(i\in I),\quad \sum_{i\in I}p_i=1.
\]
It was also asserted in \cite{LewGinTul} that $f_{i,j}<\max_k(f_{i,k}+f_{k,j})$; the proof is valid 
under an additional condition $f_{i,j}>\max\{f_{i,i},f_{j,j}\}$.

A subset $I$ meeting the condition \eqref{King2'} is a candidate to be a support set
of a local maximum of $\bold p^T \bold f\,\bold p$. (We will use a term {\it K-set\/} for such sets $I$.) 
Kingman \cite{Kin1} posed a problem of analyzing these potential
supports in a {\it typical\/} case, i.e. when $f_{i,j}$ are {\it i.i.d. random variables\/} with range $[0,1]$.
For the case when $f_{i,j}$ are $[0,1]$-uniform, he proved that with high probability (w.h.p.), i.e. with probability
approaching one, 
$\max |I|\le 2.49 n^{1/2}$: so ``the largest stable polymorphism will contain at most of the order of $n^{1/2}$
alleles''. The key tool was the bound $\pr(D_I)\le \frac{1}{r!}$, $r:=|I|$, where $D_I$ is the event in \eqref{King2'}. He
found that, for a (non-biological) exponential distribution on $[0,\infty)$, $\pr(D_I)=\Bigl(\frac{2}{r+1}\bigr)^r
\ll\frac{1}{r!}$ and the constant $2.49$ got reduced to $2.14$. 

Haigh \cite{Hai1}, \cite{Hai2} established the counterparts of some of Kingman's results for the case of a non-symmetric
payoff matrix. For instance, he proved that for the density $e^{-x}/\sqrt{\pi x},\,(x>0),$ of $\chi_1^2$ distribution, with high probability, no evolutionarily stable strategy has support of size exceeding $1.64 n^{2/3}$. Kontogiannis and Spirakis \cite{KonSpi}
used the technique from Haig \cite{Hai2} to  resolve the cases of uniform distribution and standard normal distribution
left open there.

Recently, and independently of the work cited above, Chen and Peng \cite{ChePen} studied, in an operations research context of the random quadratic optimization problems, the probability of the events quite similar to, but different from $D_I$. The probability bounds include $\frac{2^r}{(r+1)!}$ (general continuous distribution), and $\frac{2^r}{(r+1)^r}$ (uniform distribution), $\frac{2^r}{(r-1)^r}$ (exponential distribution). 

In \cite{Kin1} Kingman suggested that it should be interesting
``to carry out a comparative analysis for other distributions of the $f_{i,j}$'', and conjectured, in \cite{Kin3},
that ``for every continuous distribution $F$ of $f$, there is a finite $\beta(F)=\lim_{r\to\infty}\bigl\{r!P(D_I)\}^{1/r}\,''$. Whenever this limit exists, $\max |I|\le
2.49\beta(F) n^{1/2}$ w.h.p.; in general, $\max |I|\le
4.98 n^{1/2}$ w.h.p. 

In this paper we consider a relatively broad class of the distributions $F$, meeting the conditions:  {\bf (I)\/} $F(x)$ has
a differentiable positive density $g(x)$, $x\in [0,1]$, such that  $g'(x)\le 0$, and {\bf (II)\/} the {\it hazard ratio\/}  $\la(x):=\frac{g(x)}{1-F(x)}$ is increasing with $x$. The non-increasing linear density
$g_c(x)=\frac{1-cx}{1-c/2}$, $c\in [0,1]$ ($g_0(x)\equiv 1$) meets these constraints, and so does $g(x)=\frac{ce^{-cx}}{1-e^{-c}}$, the density of the negative exponential distribution {\it conditioned\/} on $[0,1]$.

For $F$ meeting the conditions {\bf (I)\/} and {\bf (II)\/}, we prove that 
\begin{equation}\label{add0}
\left(\frac{2}{r+1}\right)^r\le\! \pr(D_I)\le \frac{r^r}{\binom{r}{2}^{(r)}}\le \frac{e}{2}\left(\frac{2}{r}\right)^r.
\end{equation}
In combination with Kingman's analysis of the exponential distribution on $[0,\infty)$, it follows from \eqref{add0} that
for {\it every\/} $F$ meeting the constraints above, we have $\max |I|\le 2.14 n^{1/2}$ with high probability. We
see also that, for every $F$ in question,
\[
\lim_{r\to\infty}\bigl\{r!\!\pr(D_I)\bigr\}^{1/r}=:\beta(F)=\frac{2}{e},
\]
proving not only that $\beta(F)$ exists, but also that $\beta(F)$ does not depend on $F$ in this class. 
This lends a certain support to Kingman's conjecture, \cite{Kin3}, that $\lim_{r\to\infty} \bigl\{r! \!\pr(D_I)\!\bigr\}^{1/r}$ exists for every continuous $F$.

Suppose we restrict our attention to the {\it minimal\/} K-sets $I$, i.e.  such that there is no $J\subset I$, $(|J|\ge 2)$, which supports a local equilibrium $\bold p=\{p_i\}_{i\in J}$, meeting the top two conditions in \eqref{King1,2,3}. 
Let $\mathcal D_{I}$ be the corresponding event. For the distributions $F$ from the class described above, we prove that
\begin{equation}\label{mathcalD}
\pr(\mathcal D_{I})\le 2^{-r^2/2}\left(\frac{4e}{r}\right)^{r/2}\!\!\exp\bigl(\Theta(r^{1/3})\bigr),\quad r:=|I|.
\end{equation}

Continuing with $\pr(D_I)$, suppose that, in addition, $g^{(3)}(0)$ exists.  Then
\begin{equation}\label{add1}
\pr(D_I)=\bigl(1+O(r^{-\sigma})\bigr) \!\left(\frac{2}{r}\right)^r\!\!\exp\Bigl(\frac{g'(0)}{g^2(0)}\Bigr),\quad\forall\,
\sigma<1/3,
\end{equation}
and if $|I_1|=|I_2|=r$, $|I_1\cap I_2|=k$, then 
\begin{equation}\label{add2}
\pr(D_{I_1}\cap D_{I_2}) =O(\!\!\pr(r,k)),\,\, \pr(r,k):=r^{-6}\!\left(\frac{2}{r}\right)^{2(r-k-1)}\!\!\left(\!\frac{2}{2r-k}\!\right)^{k-1}\!\!,
\end{equation}
uniformly for $r\ge 2$ and $k\in [1,r-1]$.

Let $X_{n,r}$ be the total number of K-sets of $[n]$ of cardinality $r$. We already know
that  w.h.p. $X_{n,r}=0$ for $r>2.14 n^{1/2}$, and that $\ex\bigl[X_{n,r}\bigr]\to\infty$ for every $r<2.14 n^{1/2}$. We use the estimates \eqref{add0}, \eqref{add1} and \eqref{add2}  to show that
\begin{equation*}
\frac{\text{Var }(X_{n,r})}{\ex^2\bigl[X_{n,r}\bigr]}=O(n^{-2/3}),\quad 2\le r\le r(n):=\lceil 2n^{1/3}\rceil.
\end{equation*}
It follows that
\[
\pr\left(\bigcap_{\rho=2}^{r(n)}\left\{\Bigl|\frac{X_{n,\rho}}{\ex\bigl[X_{n,\rho}\bigr]}-1\Bigr|\le n^{-1/6+\eps}\right\}\right)=1-O(n^{-2\eps}),\quad \eps<1/6,
\]
i.e. w.h.p. the counts of the K-sets of size $r$ ranging from $2$ to $r(n)$ are uniformly asymptotic to their expected values. In particular, setting $L_n=\max\{\rho: X_{n,\rho}>0\}$, we have $\pr(L_n> 2n^{1/3})\to 1$, i.e. w.h.p.  
the size of the largest potential support of a local maximum is sandwiched between
$2n^{1/3}$ and $2.14 n^{1/2}$. 

We cannot rule out a possibility that, with high probability, the actual supports of local maxima are considerably smaller.
In fact, we use the bound \eqref{mathcalD} to show that, with probability $>1-n^{-a}$, $(\forall\,a>0)$, there is no $K$-set of cardinality $>(2/3)\log_2 n$ that contains, properly, a non-trivial support of a local equilibrium. Complementing this claim, 
we show that, with high
probability, the number of $K$-sets of size $< 0.5\log_2 n$ that do not contain the size $2$ supports of local equilibriums
is super-polynomially large.

The already cited paper \cite{ChePen} was preceded by Chen, Peng and Zhang \cite{ChePenZha}; both papers studied the likely behavior of an {\it absolute\/} minimum of a random quadratic form $\bold x^T\,Q\bold x$ for $\bold x\in \Delta_n$. Under the condition
that the elements of $Q$ are i.i.d. random variables with a  c.d.f. $F$ concave on its support, the support size 
 of the absolute minimum point was shown to be {\it bounded\/} in probability, with its distribution
tail decaying exponentially fast. In particular, it followed that, for $f_{i,j}$ uniform or {\it positive}-exponential on $[0,1]$, the absolute
{\it maximum\/} of $\bold p^T\,\bold f\,\bold p$ is attained at a point of $\Delta_n$ with $N$, the number of positive components, satisfying $\pr(N\ge k)=O(\rho^k)$, $k>0$, as $n\to\infty$. 

In view of all this information, it is tempting to conjecture that---for $f_{i,j}$ meeting the conditions {\bf (I)\/} and {\bf (II)\/}---the size of the largest support of a {\it local\/} maximum of $\bold p^T\,\bold f\,\bold p$ is, with high probability, of (poly)logarithmic order.
 
\section{Proofs} \subsection{Estimate of $\pr(D_I)$}
\begin{Theorem}\label{pr(D)<} Suppose that $F$ (i) has a positive, non-increasing, differentiable density
$g$, and (2) has a non-decreasing hazard ratio $\la(x)=\frac{g(x)}{1-F(x)}$. Then, with  $a^{(b)}:=a(a+1)\cdots (a+b-1)$, we have
\begin{equation}\label{pr(D)<c(r)}
\left(\frac{2}{r+1}\right)^r\le \pr(D_I)\le \frac{r^r}{\binom{r}{2}^{(r)}}\le \frac{e}{2}\left(\frac{2}{r}\right)^r.
\end{equation}
\end{Theorem}
In the special case of the uniform density $g(x)\equiv 1$, this bound improves Kingman's 
bound $\pr(D_I)\le \frac{1}{r!}$. It also shows that, for all $F$ meeting the conditions (i) and (ii),
\[
\lim_{r\to\infty} \bigl\{r! \!\pr(D_I)\!\bigr\}^{1/r} =\frac{2}{e}.
\]

\begin{proof} As in \cite{Kin1}, the probability of $D_I$, conditioned on $\{f_{i,i}=x_i,\,i\in I\}$, is
\[
\prod_{(i,j)}\!\!\pr\!\left(\!f\ge\frac{x_i+x_j}{2}\right)=\prod_{(i,j)}\!\left(\!1-F\left(\frac{x_i+x_j}{2}\right)\!\right),
\]
where $i\neq j$.The function $1-F(x)$ is log-concave, since
\[
\frac{d}{dx}\log(1-F(x))=-\frac{g(x)}{1-F(x)}=-\la(x)
\]
is decreasing with $x$. {\bf (a)\/} Lower bound. By Jensen inequality,
\begin{align*}
\prod_{(i,j)}\!\left(\!1-F\left(\frac{x_i+x_j}{2}\right)\!\right)&\ge \prod_{(i,j)}(1-F(x_i))^{1/2}(1-F(x_j))^{1/2}\\
&=\prod_{i=1}^r(1-F(x_i))^{(r-1)/2}.
\end{align*}
Consequently
\begin{equation*}
\pr(D_I)\ge \idotsint\limits_{\bold x\in [0,1]^r}\prod_{i=1}^r(1-F(x_i))^{(r-1)/2}\prod_{i=1}^r g(x_i)\, dx_i,
\end{equation*}
and, switching to the variables $y_i=F(x_i)$,
\begin{align*}
\pr(D_I)&\ge \idotsint\limits_{\bold y\in [0.1]^n}\prod_{i=1}^r (1-y_i)^{(r-1)/2}\,d\bold y\\
&=\left(\int_{0}^1(1-y)^{(r-1)/2}\,dy\right)^r=\left(\frac{2}{r+1}\right)^r.
\end{align*}
{\bf (b)\/} Upper bound. Again by Jensen inequality, denoting $s=\sum_i x_i$ we have
\begin{equation}\label{prod}
\begin{aligned}
&\prod_{(i,j)}\!\!\left(\!1-F\!\!\left(\frac{x_i+x_j}{2}\right)\!\!\right)\!=\!
\exp\Biggl[\!\binom{r}{2}\sum_{(i,j)}\!\frac{1}{\binom{r}{2}}\log\Bigl(1-F\Bigl(\frac{x_i+x_j}{2}\Bigr)\Bigr)\!\Biggr]\\
&\qquad\quad\le\exp\Biggl[\!\binom{r}{2}\log\Bigl(1-F\Bigl(\frac{1}{r(r-1)}\sum_{(i,j)}(x_i+x_j)\Bigr)\Bigr)\!\Biggr]\\
&\qquad\quad=\exp\left[\!\binom{r}{2}\log\left(1-F\left(\frac{s}{r}\right)\right)\!\right]
=\Bigl(\!1-F\left(\frac{s}{r}\right)\!\Bigr)^{\binom{r}{2}}.
\end{aligned}
\end{equation}
Consequently
\begin{equation}\label{P(D)<gen}
\pr(D_I)\le \idotsint\limits_{\bold x\in [0,1]^r}\Bigl(\!1-F\left(\frac{s}{r}\right)\!\Bigr)^{\binom{r}{2}}\prod_{i\in I} g(x_i)\, dx_i.
\end{equation}
Again change the variables of integration, setting $y_i=F(x_i)$, so that  $x_i=F^{-1}(y_i)$,
and $s=\sum_{i\in I}F^{-1}(y_i)$. Now
\[
\frac{d^2}{dy^2}\,F^{-1}(y)=-\frac{g'(x)}{g(x)^3}\ge 0,
\]
implying that $F^{-1}(y)$ is convex. Therefore, for each $t\le r$, we have
\[
r^{-1}\min\Bigl\{\sum_{i\in I}F^{-1}(y_i): \sum_{i\in I}y_i=t\Bigr\}=F^{-1}\Bigl(\frac{t}{r}\Bigr).
\]
Hence
\begin{multline*}
\max\Bigl\{\!1-F\Bigl(\!r^{-1}\sum_{i\in I}x_i\Bigr)\!:\!\sum_{i\in I}y_i=t\Bigr\}\\
=1-\min\Bigl\{\!F\Bigl(\!r^{-1}\sum_{i\in I} F^{-1}(y_i)\!\Bigr):\sum_{i\in I}y_i=t\Bigr\}\\
=1-F\Bigl(\!r^{-1}\min\Bigl\{\sum_{i\in I} F^{-1}(y_i):\sum_{i\in I} y_i=t\Bigr\}\Bigr)\\
=1-F\Bigl(F^{-1}\Bigl(\frac{t}{r}\Bigr)\Bigr)=1-\frac{t}{r}.
\end{multline*}
So \eqref{P(D)<gen} yields
\[
\pr(D_I)\le \idotsint\limits_{\bold y\in [0,1]^r}\Bigl(1-\frac{t}{r}\Bigr)^{\binom{r}{2}}\prod_{i\in I}  dy_i.
\]
Since
\[
 \idotsint\limits_{\sum_i y_i\le t}\prod_{i\in I}dy_i=\frac{t^r}{r!},
 \]
we conclude that
\begin{align*}
\pr(D_I)&\le \int_0^r \Bigl(1-\frac{t}{r}\Bigr)^{\binom{r}{2}}\frac{t^{r-1}}{(r-1)!}\,dt\\
&=\frac{r^r}{(r-1)!}\int_0^1(1-\tau)^{\binom{r}{2}}\tau^{r-1}\,d\tau\\
&=\frac{r^r}{(r-1)!}\cdot\frac{\binom{r}{2}!(r-1)!}{\left(\binom{r}{2}+r\right)!}
=\frac{r^r}{\binom{r}{2}^{(r)}}.
\end{align*}
\end{proof}

\begin{Theorem}\label{cond(iii)} Suppose that, in addition to conditions (i), (ii),  we have (iii):  $g^{(3)}(0)$ exists. Then 
\[
\pr(D_I)=\bigl(1+O(r^{-\sigma})\bigr) \!\left(\frac{2}{r}\right)^r\!\!\exp\Bigl(\frac{g'(0)}{g^2(0)}\Bigr),
\]
for every $\sigma<1/3$.
\end{Theorem}

To prove this claim, we shrink, in steps, the cube $[0,1]^n$ to a subset $C^*$ in such a way that (a)
the integral of the product of $1-F\left(\frac{x_i+x_j}{2}\right)$ over $C^*$ sharply approximates that
over $[0,1]^n$, and (b) the product itself admits a manageable approximation on $C^*$.

Given $C\subset [0,1]^n$, denote
\[
\pr_C(D_I)=\idotsint\limits_{\bold x\in C}\prod_{(i\neq j)}\!\left(\!1-F\left(\frac{x_i+x_j}{2}\right)\!\right)\,d\bold x.
\]
\begin{Lemma}\label{C1} Let 
\[
C_1:=\left\{\bold x\in [0,1]^n: \left|\sum_{i=1}^r F(x_i)-2\right|\le r^{-1/3}\right\}. 
\]
Then
\[
\pr(D_I)-\pr_{C_1}(D_I)\le \left(\frac{2}{r}\right)^r\cdot \exp\left(-\frac{r^{1/3}}{10}\right).
\]
\end{Lemma}
\begin{proof} Let $\tau_{1,2}=\frac{2}{r}\mp r^{-4/3}$. From the proof of Theorem \ref{pr(D)<} it follows that 
\begin{align*}
\pr(D_I)-\pr_{C_1}(D_I)&\le \frac{r^r}{(r-1)!}\int\limits_{\tau\in [\tau_1,\tau_2]^c}(1-\tau)^{\binom{r}{2}}\tau^{r-1}\,d\tau.
\end{align*}
The (logconcave) integrand attains its maximum at $\tau_{\text{max}}=\frac{2}{2+r}\in [\tau_1,\tau_2]$, and
\[
\max\left\{\frac{d^2}{d\tau^2}\Bigl(\log (1-\tau)^{\binom{r}{2}}\tau^{r-1}\Bigr):\,\tau\in [\tau_1,\tau_2]
\right\}\le-\frac{r^3}{4.1}.
\]
Therefore the integral is at most
\[
\left(1-\frac{2}{2+r}\right)^{\binom{r}{2}}\left(\frac{2}{2+r}\right)^{r-1}\cdot \exp\left(-\frac{r^{1/3}}{9}\right),
\]
so that 
\begin{align*}
\pr(D_I)-\pr_{C_1}(D_I)&\le \frac{r^r}{(r-1)!}\left(1-\frac{2}{2+r}\right)^{\binom{r}{2}}\!\left(\frac{2}{2+r}\right)^{r-1}\!\cdot \exp\left(-\frac{r^{1/3}}{9}\right)\\
&\le \left(\frac{2}{r}\right)^r\cdot \exp\left(-\frac{r^{1/3}}{10}\right).
\end{align*}
\end{proof}
Next
\begin{Lemma}\label{C2} Let 
\[
C_2:=\left\{\bold x\in C_1: \max_i \frac{F(x_i)}{\sum_j F(x_j)}\le k\frac{\log r}{r}\right\}, \quad (k>1).
\]
Then
\[
\pr_{C_1}(D_I)-\pr_{C_2}(D_I)\le \left(\frac{2}{r}\right)^r\cdot r^{-\alpha},\quad \forall\,\alpha<k-1.
\]
\end{Lemma}
\begin{proof} Similarly to the proof of Lemma \ref{C1},
\[
\pr_{C_1}(D_I)-\pr_{C_2}(D_I)\le \idotsint\limits_{\max\frac{y_i}{t}>k\frac{\log r}{r}}
\Bigl(1-\frac{t}{r}\Bigr)^{\binom{r}{2}}\prod_{i\in I}  dy_i.
\]
Introduce $L_1,\dots,L_r$ the lengths of the consecutive subintervals of $[0,1]$ obtained by sampling
uniformly at random $r-1$ points in $[0,1]$. By Lemma 1 in \cite{Pit0}, the integral above is at most
\begin{align*}
\pr\Bigl(\max L_i\ge k\frac{\log r}{r}\Bigr)\int_0^r &\Bigl(1-\frac{t}{r}\Bigr)^{\binom{r}{2}}\frac{t^{r-1}}{(r-1)!}\,dt\\
&=\pr\Bigl(\max L_i\ge k\frac{\log r}{r}\Bigr)\frac{r^r}{\binom{r}{2}^{(r)}}.
\end{align*}
And, introducing $U_1,\dots, U_{r-1}$ the independent $[0,1]$-Uniforms, the probability factor is at most
\begin{align*}
r\!\pr\Bigl(L_1\ge k\frac{\log r}{r}\Bigr)&=r\!\pr\Bigl(\min_i U_i \ge k\frac{\log r}{r}\Bigr)\\
&=r\Bigl(1-k\frac{\log r}{r}\Bigr)^{r-1}\le r\exp\Bigl(-(r-1)k\frac{\log r}{r}\Bigr).
\end{align*}
\end{proof}
One more reduction step defines the final
\begin{equation}\label{more}
C^*=\left\{\bold x\in C_2:\left|\frac{r}{2}\frac{\sum_iF^2(x_i)}{\left(\sum_j F(x_j)\right)^2}-1\right|\le r^{-\sigma}
\right\},\quad \sigma<1/3.
\end{equation}
\begin{Lemma}\label{C*}
\[
\pr_{C_2}(D_I) -\pr_{C^*}(D_I)\le \left(\frac{2}{r}\right)^r\cdot \exp\bigl(-0.5r^{1/3-\sigma}\bigr).
\]
\end{Lemma}
\begin{proof} Once again like the proofs of Lemmas \ref{C1}, \ref{C2},
\begin{align*}
&\pr_{C_2}(D_I) -\pr_{C^*}(D_I)\le\idotsint\limits_{\left|\frac{r}{2}\frac{\sum_i y_i^2}{\left(\sum_j y_j\right)^2}-1\right|> r^{-\sigma}}\!\!\!\!\!\left(1-\frac{t}{r}\right)^{\binom{r}{2}}\prod_{i\in I}dy_i\\
&\le \pr\Bigl(\left|\frac{r}{2}\sum_i L_i^2-1\right|>r^{-\sigma}\Bigr)\frac{r^r}{\binom{r}{2}^{(r)}}
\le \left(\frac{2}{r}\right)^r\cdot \exp\bigl(-\Theta(r^{1/3-\sigma})\bigr),
\end{align*}
as the probability is  at most $\exp\bigl(-\Theta(r^{1/3-\sigma})\bigr)$, (see Lemma 3.2 in \cite{Pit1}).
\end{proof}
{\bf Note.\/} A key to the proof of that Lemma 3.2 was a classic fact that $(L_1,\dots,L_r)$ and
$(\sum_i W_i)^{-1}(W_1,\dots, W_r)$, ($W_j$ being i.i.d. Exponentials), are equidistributed, Feller \cite{Fel}. While both of the distribution tails of $\sum_i W_j$ decay exponentially, for the right tail of $\sum_j W_j^2$ we could prove only the bound $e^{-\Theta(r^{\delta})}$, $\delta<1/3$. The obstacle here is that $\ex\bigl[e^{zW^2}\bigr]=\infty$ for every $z>0$.

Combining Lemmas \ref{C1}, \ref{C2} and \ref{C*}, we obtain
\begin{Corollary}\label{C1,C2}
\[
\pr(D_I)-\pr_{C^*}(D_I)\le \left(\frac{2}{r}\right)^r\cdot r^{-\alpha},\quad \forall\,\alpha<k-1.
\]
\end{Corollary}

For $\bold x\in C^*$, we have $\max_i F(x_i)\le \frac{3k\log r}{r}\to 0$, which implies that
$\max_i x_i =O(r^{-1}\log r)\to 0$. For $x=O(r^{-1}\log r)$, we have
\begin{align*}
F(x)&=g(0)x+\frac{1}{2}g^{(1)}(0)x^2 +O(x^3)\\
&=g(0)x+\frac{1}{2}g^{(1)}(0)x^2 +O(r^{-3}\log^3 r).
\end{align*}
So
\[
\log(1-F(x)) =-g(0)x-\frac{g'(0)+g^2(0)}{2}x^2+O(r^{-3}\log^3 r),
\]
and with a bit of algebra
\begin{equation}\label{gamma=}
\begin{aligned}
&\log\left(1-F\left(\frac{x_i+x_j}{2}\right)\right)-\frac{\log(1-F(x_i))+\log(1-F(x_j))}{2}\\
&=\frac{g'(0)+g^2(0)}{8}(x_i-x_j)^2+O(r^{-3}\log^3 r)\\
&=\gamma(F(x_i)-F(x_j))^2+O(r^{-3}\log^3 r),\quad \gamma:=\frac{g'(0)+g^2(0)}{8g^2(0)}.
\end{aligned}
\end{equation}
Therefore
\begin{align*}
\prod_{(i,j)}\log\left(1-F\left(\frac{x_i+x_j}{2}\right)\right)&=\exp\left(\frac{r-1}{2}\sum_i\log(1-F(x_i))\right)\\
&\times\exp\left(\gamma\sum_{(i,j)}(F(x_i)-F(x_j))^2+O(r^{-1}\log r)\right),
\end{align*}
where
\begin{align*}
\frac{r-1}{2}\sum_i\log(1-F(x_i)&=-\frac{r-1}{2}\sum_i \Bigl(F(x_i)+\frac{F^2(x_i)}{2}\Bigr)+O(r^{-1}\log r),\\
\sum_{(i,j)}(F(x_i)-F(x_j))^2&=r\sum_iF^2(x_i)-\left(\sum_iF(x_i)\right)^2.
\end{align*}
Hence on $C^*$ (see \eqref{more})
\begin{align*}
&\prod_{(i,j)}\log\left(\!1-F\left(\frac{x_i+x_j}{2}\right)\!\right)
=\exp\left(\!-\frac{r-1}{2}\sum_iF(x_i)-\gamma\Bigl(\sum_iF(x_i)\Bigr)^2\right.\\
&\qquad\qquad\qquad\qquad\qquad\qquad\left.+ \Bigl(\!-\frac{r-1}{4}+\gamma r\!\Bigr)\sum_iF^2(x_i)+O(r^{-1}\log r)\!\right)\\
&\qquad\quad=\exp\left(-\frac{r-1}{2}\sum_iF(x_i)+\Bigl(2\gamma-\frac{1}{2}\Bigr)\left(\sum_iF(x_i)\right)^2+O(r^{-\sigma})\right)\\
&\qquad\qquad\qquad =\exp\Bigl(-\frac{r}{2}\sum_iF(x_i)+\frac{g'(0)}{g^2(0)}+O(r^{-\sigma})\Bigr);
\end{align*}
for the last equality we used the definition of $\gamma$ in \eqref{gamma=}. 

Switching to the variables $y_i=F(x_i)$ and denoting $t=\sum_iy_i$, we obtain then
\begin{align*}
&\qquad\qquad\pr_{C^*}(D_I)=\idotsint\limits_{\bold y\in \mathcal C^*}\exp\Bigl(-\frac{r}{2}t+\frac{g'(0)}{g^2(0)}+O(r^{-\sigma})\Bigr)\,d\bold y,\\
\mathcal C^*&:=\left\{\bold y\ge\bold 0: |t-2|\le r^{-1/3},\,\max_i \frac{y_i}{t}\le k\frac{\log r}{r},\,
\left|\frac{r}{2t^2}\sum_iy_i^2-1\right|\le r^{-\sigma}\right\}.
\end{align*}
Notice that on $\mathcal C^*$ we  have $\max_i y_1\to 0$, so that the omitted condition $\max_i y_i\le 1$
would have been superfluous. By Lemma 3.1 in \cite{Pit1},
\begin{align*}
&\qquad\qquad\qquad\qquad\qquad\qquad\idotsint\limits_{y\in \mathcal C^*}e^{-\frac{rt}{2}}\,d\bold y\\
&=\!\!\!\int\limits_{|t-2|\le \frac{1}{r^{1/3}}} \!\! \frac{e^{-\frac{rt}{2}}t^{r-1}}{(r-1)!} \pr\!\left(\!\max L_i\le \min\Bigl(t^{-1}, \frac{k\log r}{r}
\Bigr),\,\Bigl|\frac{r}{2}\sum_i L_i^2-1\Bigr|\le r^{-\sigma}\!\right)\,dt\\
&\qquad=\pr\!\left(\!\max L_i\le \frac{k\log r}{r},\,\Bigl|\frac{r}{2}\sum_i L_i^2-1\Bigr|\le r^{-\sigma}\!\right)
\!\int\limits_{|t-2|\le \frac{1}{r^{1/3}}} \frac{e^{-\frac{rt}{2}}t^{r-1}}{(r-1)!} \,dt.
\end{align*}
From  Lemma \ref{C2} and Lemma \ref{C*}, and their proofs, we know that the probability factor is
at least $1-r^{-\alpha}$, $\forall\,\alpha<k-1$. Furthermore, the integral equals
\begin{align*}
&\int\limits_0^{\infty}\frac{e^{-\frac{rt}{2}}t^{r-1}}{(r-1)!} \,dt\,\,\,-\!\!\int\limits_{|t-2|>\frac{1}{r^{1/3}}}\!\!\!\!\!\!\!
\frac{e^{-\frac{rt}{2}}t^{r-1}}{(r-1)!} \,dt\\
&=\left(\frac{2}{r}\right)^r-\frac{\left(\frac{2}{r}\right)^r}{(r-1)!}\int\limits_{|\tau-r|>\frac{r^{2/3}}{2}}\!\!\!\!\!\!\!
e^{-\tau}\tau^{r-1}\,d\tau,
\end{align*}
and, by Chebyshev's inequality,
\begin{align*}
\frac{1}{(r-1)!}\int\limits_{|\tau-r|>\frac{r^{2/3}}{2}}\!\!\!\!\!\!\!e^{-\tau}\tau^{r-1}\,d\tau&\le
\pr\Bigl(\bigl|\text{Poisson }(r-1)-(r-1)\bigr|>\frac{r^{2/3}}{3}\Bigr)\\
&\le\frac{9(r-1)}{r^{4/3}}\le 9 r^{-1/3}.
\end{align*}
So 
\[
\int\limits_{|t-2|\le \frac{1}{r^{1/3}}} \frac{e^{-\frac{rt}{2}}t^{r-1}}{(r-1)!} \,dt=\bigl(1+O(r^{-1/3})\bigr)\left(\frac{2}{r}\right)^r.
\] 
Consequently
\[
\pr_{\mathcal C^*}(D_I)=\bigl(1+O(r^{-\sigma})\bigr)\left(\frac{2}{r}\right)^r \exp\Bigl(\frac{g'(0)}{g^2(0)}\Bigr),
\]
for every $\sigma<1/3$. Combining this estimate with Corollary \ref{C1,C2}, we complete the proof
of Theorem \ref{cond(iii)}.
\subsection{Estimate of $\pr(D_{I_1}\cap D_{I_2})$} Let $I_1,\,I_2\subset [n]$, $|I_j|=r$. If $I_1\cap I_2=\emptyset$, then the events
$D_{I_1}$ and $D_{I_2}$ are independent and so (by Theorem \ref{cond(iii)})
\begin{equation}\label{I1I2dis}
\pr(D_{I_1}\cap D_{I_2})=\!\pr(D_{I_1})\cdot\!\pr(D_{I_2})=\bigl(1+O(r^{-\sigma})\bigr)\left(\frac{2}{r}\right)^{2r} \!\!\!\exp\Bigl(2\frac{g'(0)}{g^2(0)}\Bigr).
\end{equation}
Consider the case $|I_1\cap I_2|=k\in [1,r-1]$. By symmetry, we can assume that $I_1=\{1,\dots,r\}$, and 
$I_2=\{r-k+1,\dots, 2r -k\}$. The probability of $D_{I_1}\cap D_{I_2}$, conditioned on the event $\{F_{i,i}=x_i: 1\le i\le 2r-k\}$, is
\begin{equation}\label{condprob=}
\begin{aligned}
\Psi(\bold x)&=\prod_{(i\neq j)\atop i,\,j\le r}\!\left(\!1-F\left(\frac{x_i+x_j}{2}\right)\!\right)
\prod_{r<i\le 2r-k\atop \,r-k+1\le j<r}\!\!\left(\!1-F\left(\frac{x_i+x_j}{2}\right)\!\right)\\
&\times\prod_{(i\neq j)\atop r\le i,\,j\le 2r-k}\!\!\!\left(\!1-F\!\left(\frac{x_i+x_j}{2}\right)\!\right).
\end{aligned}
\end{equation}
The three products contain, respectively,  $\binom{r}{2}$, $(r-k)(k-1)$ and $\binom{r-k+1}{2}$ factors.
The total number of the factors is
\[
N(r,k)=\binom{r}{2}+ (r-k)(k-1)+\binom{r-k+1}{2}.
\] 
Now
\begin{align*}
&\sum_{(i\neq j)\atop 1\le i,\,j\le r}\frac{x_i+x_j}{2}=\frac{r-1}{2}\sum_{i=1}^r x_i, \,\,
\sum_{(i\neq j)\atop r\le i,\,j\le 2r-k}\frac{x_i+x_j}{2}=\frac{r-k}{2}\sum_{i=r}^{2r-k}x_i,\\
&\sum_{r<i\le 2r-k\atop r-k+1\le j\le r}\frac{x_i+x_j}{2}=\frac{k-1}{2}\sum_{i=r+1}^{2r-k}x_i+
\frac{r-k}{2}\sum_{j=r-k+1}^{r-1}\!\! x_j.
\end{align*}
The total sum of the fractions $\frac{x_i+x_j}{2}$ is $\frac{r-1}{2}s_1+\frac{2r-k-1}{2}s_2+\frac{r-1}{2}s_3$, where
\[
s_1=\sum_{i=1}^{r-k} x_i, \,\,\, s_2=\!\!\sum_{i=r-k+1}^{r}\!\!\!x_i,\,\,\,s_3=\!\sum_{i=r+1}^{2r-k}x_i,
\]
and the sum of the coefficients $\alpha_i$ by $x_i$ in the sum of those fractions is $N(r, k)$.
By log-concavity of $1-F(x)$,
\begin{equation*}
\Psi(\bold x)\le \left(\!1-F\!\left(\frac{\frac{r-1}{2}s_1+\frac{2r-k-1}{2}s_2+\frac{r-1}{2}s_3}{N(r, k)}\!\right)\!\right)^{N(r, k)}.
\end{equation*}
As in the proof of Theorem \ref{pr(D)<}, introduce $y_i=F(x_i)$, $1\le i\le 2r-k$, so that
\[
s_1=\sum_{i=1}^{r-k} F^{-1}(y_i), \,\,\, s_2=\!\!\sum_{i=r-k+1}^{r}F^{-1}(y_i),\,\,\,s_3=\!\sum_{i=r+1}^{2r-k}
F^{-1}(y_i).
\]
Given $t_1,\,t_2,\,t_3$, 
by convexity of $F^{-1}$, we have
\begin{align*}
&\min\left\{\sum_{i=1}^{2r-k}\frac{\alpha_i}{N(r,k)}F^{-1}(y_i): \sum_{i=1}^{r-k}y_i=t_1,\,\sum_{i=r-k+1}^{r}\!\!y_i=t_2,\,\sum_{i=r+1}^{2r-k}y_i=t_3\right\}\\
&\qquad\quad\ge F^{-1}\!\left(\frac{\frac{r-1}{2}t_1}{N(r,k)}+\frac{\frac{2r-k-1}{2}t_2}{N(r,k)}+\frac{\frac{r-k}
{2}t_3}{N(r,k)}\right).
\end{align*}
Consequently
\begin{align*}
&\Psi(\bold x)\le\Psi^*(\bold t):=\! \left(\!1-\frac{(r-1)t_1+(2r-k-1)t_2 +(r-1)t_3}{2N(r, k)}\!\right)^{\!N(r,k)},\\
&\qquad t_1:=\sum_{i=1}^{r-k}F(x_i),\,\,t_2:=\!\!\sum_{i=r-k+1}^r\!\!\!F(x_i),\,\,t_3:=\sum_{i=r+1}^{2r-k} \!F(x_i).
\end{align*}
Therefore
\begin{equation}\label{P(cap)<(t)}
\begin{aligned}
&\pr(D_{I_1}\cap D_{I_2})=\idotsint\limits_{\bold x\in [0,1]^{2r-k}}\Psi(\bold x)\,d\bold x
\le\!\!\idotsint\limits_{ t_1\le r-k,\, t_2\le k,\,t_3\le r-k}\!\!\!\!\!\!\!\Psi^*(\bold t)\,d\bold y\\
&=\!\!\!\!\!\!\iiint\limits_{ t_1,\,t_3\le r-k,\, t_2\le k}\!\!\!\!\!\!\!\!\Psi^*(\bold t)
\frac{t_1^{r-k-1}}{(r-k-1)!}\frac{t_2^{k-1}}{(k-1)!}\frac{t_3^{r-k-1}}{(r-k-1)!}\,d\bold t.
\end{aligned}
\end{equation}
Introduce
\[
\tau_1=\frac{r-1}{2N(r,k)}t_1,\,\,\tau_2=\frac{2r-k-1}{2N(r,k)}t_2,\,\tau_3=\frac{r-1}{2N(r,k)}t_3.
\]
Since 
\begin{align*}
&\qquad\qquad t_1\le r-k,\,\, t_2\le k,\,\, t_3\le r-k,\\
&\frac{(r-1)(r-k)}{2N(r,k)}+\frac{(2r-k-1)k}{2N(r,k)}+\frac{(r-1)(r-k)}
{2N(r,k)}=1,
\end{align*}
we see that $\tau_1+\tau_2+\tau_3\le 1$. Switching to $\tau_j$, and denoting $N=N(r,k)$, we transform \eqref{P(cap)<(t)},
 into
\begin{align*}
&\pr(D_{I_1}\cap D_{I_2})\le \frac{\left(\frac{2N}{r-1}\right)^{r-k-1}}{(r-k-1)!}\cdot\frac{\left(\frac{2N}{2r-k-1}\right)^{k-1}}{(k-1)!}\cdot\frac{\left(\frac{2N}{r-1}\right)^{r-k-1}}{(r-k-1)!}\\
&\quad\times\iiint\limits_{\tau_1+\tau_2+\tau_3\le 1}\!\!\!\!\tau_1^{r-k-1}\tau_2^{k-1}\tau_3^{r-k-1}(1-\tau_1-\tau_2-\tau_3)^N\,d\tau_1d\tau_2d\tau_3\\
&\qquad\qquad=\frac{N!\left(\frac{2N}{r-1}\right)^{2(r-k-1)}\left(\frac{2N}{2r-k-1}\right)^{k-1}}{(N+2r-k)!}
\\
&\qquad\le
N^{-3}\left(\frac{2}{r-1}\right)^{2(r-k-1)}\!\left(\frac{2}{2r-k-1}\right)^{k-1}.
\end{align*}
(At the penultimate line we used the multidimensional extension of the beta integral, Andrews, Askey and Roy \cite{AndAskRoy}, Theorem 1.8.6.) Since $N=\Theta(r^2)$, we have then
\begin{equation}\label{P(D1D2}
\pr(D_{I_1}\cap D_{I_2}) =O(\!\!\pr(r,k)),\,\, \pr(r,k):=r^{-6}\!\left(\frac{2}{r}\right)^{2(r-k-1)}\!\!\left(\!\frac{2}{2r-k}\!\right)^{k-1}\!\!.
\end{equation}
\subsection{Likely range of the maximum size of the K-set} Introduce $L_n=\{\max |I|:
\eqref{King2'}\text{ holds}\}$. Kingman \cite{Kin1}, \cite{Kin2}, \cite{Kin3} proved that, for $F=\text{Uniform}[0,1]$,
w.h.p. $L_n\le n^{1/2}(\epsilon+o(1))$, where $\epsilon=\xi^{-1/2}(1-\xi)^{-1/2}$ and $\xi=2.49\dots$ is a positive root of $1-\xi=e^{-2\xi}$. The proof consisted of showing that $\!\pr(D_I)\le \frac{1}{|I|!}$, 
and that
\begin{equation}\label{r,s}
\pr(L_n\ge r)\le \frac{(n)_s}{(r)_s}\!\pr(D_I),\quad |I|=s\le r.
\end{equation}
This inequality sharpens the (first-order moment) bound $\pr(L_n\ge r)\le \binom{n}{r}\!\pr(D_I)$, $|I|=r$, 
by using the fact that every subset of a K-set is a K-set as well. Kingman also demonstrated that his {\it exact\/} formula $\pr(D_I)=\left(\frac{2}{r+1}\right)^r$ for the negative exponential distribution on $[0,\infty)$
implied a better bound
\begin{equation}\label{2/e}
L_n\le n^{1/2}\bigl[(2e^{-1})^{1/2}\epsilon +o(1)\bigr],\quad (2e^{-1})^{1/2}\epsilon=2.14\dots.
\end{equation} 
Now, by Theorem \ref{pr(D)<}, we have $\pr(D_I)\le \frac{e}{2}\left(\frac{2}{r}\right)^r$ for a wide class of
the densities on $[0,1]$, that includes the uniform density and the exponential density restricted to $[0,1]$.
Combining this Theorem and Kingman's proof for the exponential distribution, we obtain
\begin{Theorem}\label{L_n<2.14n^{1/2}} Under the conditions (i), (ii) of Theorem \ref{pr(D)<},
w.h.p.
\[
L_n\le n^{1/2}(2.14...+o(1)).
\]
\end{Theorem}
Armed with the bound \eqref{P(D1D2} and the bounds in Theorem \ref{pr(D)<}, we can prove a qualitatively matching lower bound.
\begin{Theorem}\label{L_n>2n^{1/3}} Let $X_{n,r}$ stand for the total number of K-sets of
cardinality $r$. Introduce $r(n)=\lceil 2n^{1/3}\rceil$. Then, under the conditions (i), (ii) and (iii) of Theorem \ref{cond(iii)}, 
\[
\pr\left(\bigcap_{\rho=2}^{r(n)}\left\{\Bigl|\frac{X_{n,\rho}}{\ex\bigl[X_{n,\rho}\bigr]}-1\Bigr|\le n^{-1/6+\eps}\right\}\right)
=1-O(n^{-2\eps}),\quad\forall\, \eps\in (0,1/6).
\]
Consequently, $\min_{r\in [2,r(n)]}X_{n,r}\to\infty$ in probability, and so
\[
\lim_{n\to\infty}\!\pr\bigl(L_n\ge 2n^{1/3}\bigr)=1.
\]
\end{Theorem} 
\begin{proof} This time we use the second-order moment approach. By Theorem \ref{pr(D)<}, for a generic set $I$ of cardinality $r\in [2,r(n)]$ we have
\[
\ex\bigl[X_{n,r}\bigr]=\binom{n}{r}\!\!\pr(D_I)\ge\frac{n^r}{2r!} \left(\frac{2}{r+2}\right)^r\ge \text{const }n^2.
\]
The total number of ordered pairs $\{I_1,I_2\}$, with $|I_1|=|I_2|=r$, $|I_1\cap I_2|=k$, is
\[
\mathcal N(r,k)=\binom{n}{r}\binom{r}{k}\binom{n-r}{r-k}.
\]
Therefore, for a pair of generic sets $I_1$, $I_2$ meeting the conditions above,
\begin{equation}\label{E_2}
\ex\bigl[(X_{n,r})_2\bigr]=\sum_{k=0}^{r-1}\mathcal N(r,k)\!\pr(D_{I_1}\cap D_{I_2}).
\end{equation}
Here $\pr(D_{I_1}\cap D_{I_2})=O(\!\!\pr(r,k))$, with $\pr(r,k)$ given in \eqref{P(D1D2}. After some elementary computations we obtain that
\begin{align*}
&\qquad\qquad\qquad\qquad\quad\max_{k\in [1,r-1]}\frac{\mathcal N(r,k+1)\!\pr(r,k+1)}{\mathcal N(r,k)\!\pr(r,k)}\\
&\le 
\frac{r^2}{2}\max_{k\in [1,r-1]}\frac{(r-k)^2}{(k+1)(2r-k-1)(n-2r+k+1)}\cdot\exp\left(\frac{k-1}{2r-k-1}\right)\\
&\qquad\qquad\qquad\qquad\quad=\frac{r^2(r-1)}{8(n-2r+2)}\le e^{8n^{-2/3}};
\end{align*}
(the second line maximum is attained at $k=1$). Consequently
\begin{align*}
\sum_{k=1}^{r-1}\mathcal N(r,k)\!\pr(r,k)&\le re^{8rn^{-2/3}}\mathcal N(r,1)\!\pr(r,1)\le 2r\mathcal N(r,1)\!\pr(r,1)\\
&\le 2\binom{n}{r}\binom{n-r}{r-1}\left(\frac{2}{r}\right)^{2r}\\
&=O\Bigl(n^{-1}\mathcal N(r,0)\!\pr^2(D_I)\bigr)=O\Bigl(\frac{r}{n}\ex^2\bigl[X_{n,r}\bigr]\Bigr).
\end{align*}
(For the last equality we used the lower bound for $\!\!\pr(D_I)$ in Theorem \ref{pr(D)<}.) 
Therefore, uniformly for $r\in [2,r(n)]$,
\begin{equation}\label{2}
\frac{\sum_{k=1}^{r-1} \mathcal N(r,k)\!\pr(r,k)}{\ex^2\bigl[X_{n,r}\bigr]}=O(n^{-2/3}).
\end{equation}
From the equations
\eqref{E_2} and \eqref{2}, and $\ex\bigl[X_{n,r}\bigr]\ge \text{const }n^2\gg n^{2/3}$, we have
\[
\frac{\ex\bigl[(X_{n,r})_2\bigr]}{\ex^2\bigl[X_{n,r}\bigr]}=1+O(n^{-2/3})\Longrightarrow 
\frac{\text{Var}(X_{n,r})}{\ex^2\bigl[X_{n,r}\bigr]}=O(n^{-2/3}).
\]
By Chebyshev's inequality,
\[
\pr\left(\left|\frac{X_{n,r}}{\ex\bigl[X_{n,r}\bigr]}-1\right|\le \delta\right)\ge 1-O\bigl(\delta^{-2}n^{-2/3}\bigr)\to 1,
\]
uniformly for all $\delta\gg n^{-1/3}$ and $r\in [2,r(n)]$. Therefore
\[
\sum_{r=2}^{r(n)}\pr\left(\Bigl |\frac{X_{n,r}}{\ex\bigl[X_{n,r}\bigr]}-1\Bigr |\ge\delta\right)
=O\bigl(\delta^{-2}n^{-1/3}\bigr)\to 0,
\]
which implies:  for $\eps\in (0,1/6)$,
\[
\pr\left(\bigcap_{r=2}^{r(n)}\left\{\Bigl|\frac{X_{n,r}}{\ex\bigl[X_{n,r}\bigr]}-1\Bigr|\le n^{-1/6+\eps}\right\}\right)
=1-O(n^{-2\eps}).
\]
\end{proof}
\subsection{Estimate of $\pr(\mathcal D_I)$} Recall that the event $\mathcal D_I$ happens iff $I$ is a $K$-set
{\it and\/} no $J\subset I$, with $|J|\ge 2$, supports a local equilibrium  $\bold p=\{p_i\}_{i\in J}>\bold 0$, ($\sum_{i\in J}p_i=1$).
\\

Let the event $D_I$ holds, so that $f_{u,v}\ge (f_{u,u}+f_{v,v})/2$ for all $u,v\in I$. So $D_J$ holds for every $J\subseteq I$.
Suppose that for some $i\neq j$ in $I$ we have $f_{i,j} > \max\{f_{i,i},\,f_{j,j}\}$. Set $J=\{i,j\}$ and
\[
p_i:=\frac{f_{i,j}-f_{j,j}}{2f_{i,j}-f_{i,i}-f_{j,j}}>0, \quad p_j=\frac{f_{i,j}-f_{i,i}}{2f_{i,j}-f_{i,i}-f_{j,j}}>0.
\]
Then $\bold p=(p_i,p_j)$ is a non-trivial local equilibrium,  and this cannot happen on the event $\mathcal D_I$. Thus
\[
\mathcal D_I\subseteq\bigcap_{(i\neq j):\, i,j\in I}\left\{\frac{f_{i,i}+f_{j,j}}{2}\le f_{i,j}\le \max\{f_{i,i}, f_{j,j}\}\right\}.
\]
Consequently we obtain
\begin{equation}\label{<}
\pr(\mathcal D_I\,|\,f_{i,i}=x_i,\,i\in I)\le\!\! \prod_{(i\neq j):\, i,j\in I}\!\left[\!F\bigl( \max\{x_i, x_j\}\bigr)-\!
F\!\left(\!\frac{x_i+x_j}{2}\!\right)\!\right].
\end{equation}
Introduce $y_i=F(x_i)$,  i.e. $x_i=F^{-1}(y_i)$, $(i\in I)$. Then $F\bigl( \max\{x_i, x_j\}\bigr)=\max\{y_i,y_j\}$,
and (since $F^{-1}(y)$ is convex),
\begin{align*}
F\left(\frac{x_i+x_j}{2}\right)&=F\left(\frac{F^{-1}(y_i)+F^{-1}(y_j)}{2}\right)\\
&\ge F\left(F^{-1}\left(\frac{y_i+y_j}{2}\right)
\right)=\frac{y_i+y_j}{2}.
\end{align*}
Therefore
\[
\pr\bigl(\mathcal D_I\,|\,f_{i,i}=x_i,\,i\in I\bigr)\le \prod_{(i\neq j):\, i,j\in I}\!\frac{|y_i-y_j|}{2},
\]
implying
\begin{equation}\label{pr<Selint}
\pr(\mathcal D_I)\le 2^{-r(r-1)/2}\idotsint\limits_{\bold y\in [0,1]^r}\prod_{(i\neq j):\, i,j\in I}\! |y_i-y_j|\,\,d\bold y,\quad
r:=|I|.
\end{equation}
Since the integral is below $1$, we see that
\begin{equation}\label{<2^{-r(r-1)/2}}
\pr(\mathcal D_I)\le 2^{-r(r-1)/2}.
\end{equation} 
Hence
\begin{Corollary}\label{minsupmax} With probability $\ge 1- n^{-a}$, $(\forall\,a>0)$, there is no $K$-set 
of cardinality $\ge r_n:=\lceil 2\log_2 n\rceil$, that  contains, properly, the support of a non-trivial local equilibrium. 
\end{Corollary}
\begin{proof} By \eqref{<2^{-r(r-1)/2}} the expected number of $K$-sets in question is, at most, of order
\[
\binom{n}{r_n} 2^{-r_n^2/2}\le \frac{1}{r_n!}\le n^{-a}, \quad\forall\, a>0.
\]
\end{proof}
We can do better though. The integral in \eqref{pr<Selint} is a special case of Selberg's remarkable
integral, \cite{AndAskRoy}, Section 8.1: in particular, for $\alpha>0,\, \beta>0, \,\gamma\ge 0$,
\begin{equation}\label{Selb}
\begin{aligned}
&\idotsint\limits_{\bold y\in [0,1]^r}\prod_{i\in I}\bigl\{y_i^{\alpha-1}(1-y_i)^{\beta-1}\bigr\}\,\prod_{(i\neq j):\, i,j\in I}\! 
|y_i-y_j|^{2\gamma}\,\,d\bold y\\
=&\prod_{j=1}^r\frac{\Gamma\bigl(\alpha+(j-1)\gamma\bigr)\,\Gamma(\bigl(\beta+(j-1)\gamma\bigr)\,\Gamma(1+j\gamma)}{\Gamma\bigl(\alpha+\beta+(r+j-2)\gamma\bigr)\,\Gamma(1+\gamma)}.
\end{aligned}
\end{equation}
So we have
\[
\pr(\mathcal D_I)\le 2^{-r(r-1)/2}\mathcal S(r),\quad \mathcal S(r):=\prod_{j=1}^r\frac{\Gamma^2\bigl(1+(j-1)/2\bigr)\,\Gamma(1+j/2)}{\Gamma\bigl(1+(r+j)/2\bigr)\,\Gamma(3/2)}.
\]
Using the Stirling formula 
\[
\Gamma(1+z)=\sqrt{2\pi z}\left(\frac{z}{e}\right)^z(1+O(z^{-1})),\quad z\to\infty,
\]
and applying the Euler summation formula to the logarithm of the resulting product, one can show that, for some constants $\eta_1$, $\eta_2$,
\begin{equation}\label{prod(j:r)approx}
\mathcal S(r)=
2^{-r^2}\exp\!\left(\eta_1r\log r+ \eta_2 r +O(\log r)\right).
\end{equation}
We have proved
\begin{Lemma}\label{pr(mathcal D)approx} There exist  constants $\eta_1^*$, $\eta_2^*$ such that
\[
\pr(\mathcal D_I)\le 2^{-\frac{3}{2}r^2}\exp\left(\eta_1^*r\log r + \eta_2^* r +O(\log r)\right),\quad r:=|I|.
\]
\end{Lemma}
So $\pr(\mathcal D_I)$ is of order $2^{-\frac{3(1+o(1))}{2}r^2}$, at most. This leads immediately to a better upper bound 
fot the maximum size of a $K$-set free of supports of local equilibriums. 
\begin{Theorem}\label{2/3} With probability $\ge 1- \exp\bigl(-\Theta(\eps \log^2 n)\bigr)$,  there is no $K$-set 
of cardinality $\ge r_n^*:=\lceil (2/3+\eps)\log_2 n\rceil$, that properly contains the support of a non-trivial local equilibrium. 
\end{Theorem}

The sharp formula  \eqref{prod(j:r)approx} allows us to show that with high probability there exist many $K$ sets of the
logarithmic size that do not contain the size $2$ supports of local equilibriums in the case when $f_{i,j}$ are uniform.

Given a set $I$, $|I|\ge 3$, let $\mathcal D_I^*$ be the event that $I$ is a $K$-set meeting the above, less
stringent, requirement. For brevity, we call such $I$ a $K^*$-set. Instead of the inequality \eqref{<}, here we have the equality 
 \begin{equation}\label{=}
\pr(\mathcal D_I^*\,|\,f_{i,i}=x_i,\,i\in I)=\!\! \prod_{(i\neq j):\, i,j\in I}\!\left[\!F\bigl( \max\{x_i, x_j\}\bigr)-\!
F\!\left(\!\frac{x_i+x_j}{2}\!\right)\!\right].
\end{equation}
For the uniform fitnesses  the RHS in \eqref{=} is the product of the $|x_i-x_j|/2$.
So, by  \eqref{Selb} and \eqref{prod(j:r)approx},
\begin{equation}\label{P(mathcalD_I^*)=}
\pr(\mathcal D_I^*)=2^{-\frac{3}{2}\rho^2}\exp\left(\eta_1^*\rho\log \rho + \eta_2^* \rho +O(\log\rho)\right),\quad \rho:=|I|.
\end{equation}
Let $X_{n,r}^*$ denote the total number of the $K^*$-sets of cardinality $r$. Then the expected number of the $K^*$-sets
of cardinality $r$ is $\ex[X_{n,r}^*]=\binom{n}{r}\pr(\mathcal D_I^*)$, $(|I|=r)$. This expectation is easily shown
to be of order $\ge \exp\bigl(\Theta(\eps\log^2 n)\bigr)$, thus super-polynomially large, if $r=\bigl[(2/3)(1-\eps)\log_2 n\bigr]$, $\eps\in (0,1)$. In fact, we are about to prove that $X_{n,r}^*$ is likely to be this large if $r< 0.5 \log_2 n$.
\begin{Theorem}\label{K*}  For $\,r=\bigl[(0.5-\eps)\log_2 n\bigr]$, ($\eps<1/4$), we have
\[
\pr\Bigl(X_{n,r}^*\ge \exp\bigl(\Theta(\eps\log^2 n)\bigr)\Bigr)\ge 1- O\bigl(n^{-2\eps+O(\log\log n/(\log n))}\bigr).
\]
\end{Theorem}
\begin{proof} We use the proof of Theorem \ref{L_n>2n^{1/3}} as a rough template. Given $0\le k\le r-1$, let 
\[
I_1=I_1(k)\equiv \{1,\dots,r\},\quad I_2=I_2(k)=\{r-k+1,\dots, 2r-k\}; 
\]
so $|I_1|=|I_2|=r$ and $|I_1\cap I_2|=k$. Then, by symmetry, 
\[
\ex\bigl[(X_{n,r}^*)_2\bigr]=\sum_{k=0}^{r-1}\mathcal N(r,k)\mathcal \pr\bigl(\mathcal D^*_{I_1(k)}\cap \mathcal D^*_{I_2(k)}\bigr),
\quad \mathcal N(r,k)=\binom{n}{r}\binom{r}{k}\binom{n-r}{r-k}.
\]
To bound $\mathcal \pr\bigl(\mathcal D^*_{I_1(k)}\cap \mathcal D^*_{I_2(k)}\bigr)$, observe that, denoting by $(i,j)$ a generic,
unordered pair $(i\neq j)$, we have
\begin{multline*}
\pr\bigl(\mathcal D_{I_1(k)}^*\cap\mathcal D_{I_2(k)}^*\,|\,f_{i,i}=x_i,\,i\in I_1\cup I_2\bigr)=\prod_{(i, j)\atop i, j\in [1,r]\cup [r-k+1,2r-k]}
\!\!\!\!\!\!\!\!\!\!\!\frac{|x_i-x_j|}{2}\\
\le 2^{-(r)_2+\binom{k}{2}}\prod_{(i,j)\atop i,j\in [1,r-k]} |x_i-x_j|\\
\times\prod_{(i,j)\atop i,j\in [r-k+1,r]} |x_i-x_j| \prod_{(i,j)\atop i,j\in [r+1,2r-k]} |x_i-x_j|.
\end{multline*}
Unconditioning and using \eqref{=}, we obtain: 
\begin{equation}\label{uncond}
\begin{aligned}
&\pr\bigl(\mathcal D_{I_1(k)}^*\cap\mathcal D_{I_2(k)}^*\bigr)
=\mathcal P^*(r,k) e^{O(\log r)},\\
\mathcal P^*(r,k)&:=2^{-(r)_2+\binom{k}{2}}\cdot 2^{-2(r-k)^2-k^2}\\
&\qquad\times\exp\Bigl[2\eta_1^*(r-k)\log (r-k)+\eta_1^*k\log k +2\eta_2^*(r-k)+\eta_2^*k \Bigr].
\end{aligned}
\end{equation}
It follows that 
\begin{align*}
\frac{\mathcal N(r,k+1)\pr\bigl(\mathcal D_{I_1(k+1)}^*\cap\mathcal D_{I_2(k+1)}^*\bigr)}
{\mathcal N(r,k)\pr\bigl(\mathcal D_{I_1(k)}^*\cap\mathcal D_{I_2(k)}^*\bigr)}&\le \frac{2^{2r}}{n} \exp\bigl(O(\log r)\bigr)\\
&\le n^{-2\eps +o(1)}\to 0,
\end{align*}
since $r\le (0.5-\eps)\log _2 n$. Consequently
\[
\frac{\ex\bigl[(X_{n,r}^*)_2\bigr]}{\mathcal N(r,0) \pr^2(\mathcal D_{I_1(0)}^*)}\le 1+n^{-2\eps +o(1)}.
\]
Since 
\begin{align*}
&\mathcal N(r,0)=\bigl(1+O(r^{2}/n)\bigr)\binom{n}{r}^2,\\
& \ex\bigl[X_{n,r}^*\bigr]=\pr(\mathcal D_{I_1^*}(0))\binom{n}{r}\ge \exp\bigl(\Theta( \log^2 n)\bigr),
\end{align*}
the Chebyshev inequality completes the proof.
\end{proof}

{\bf Acknowledgment.\/} About thirty years ago John Kingman gave a lecture on stable polymorphisms at
Stanford University. The talk made a deep, lasting impression on me. At that time Don Knuth introduced me to his
striking formula for the expected number of stable matchings via a highly-dimensional integral, \cite{Knu}. Despite the world of difference between the stable polymorphisms and the stable matchings, the multidimensional integrals expressing the probability of respective stability conditions are of a similar kind.

\end{document}